\documentclass{amsart}
\usepackage{vmargin,amsmath,amsfonts,amssymb,amsthm,xcolor,amscd,latexsym,xspace,enumerate}
\usepackage[pagebackref, colorlinks=true,linkcolor=blue,citecolor=red]{hyperref}
\usepackage{hyperref}
\usepackage{mathrsfs} 
\usepackage{pdfpages}
\usepackage{tikz-cd}
\usetikzlibrary{arrows}
\usetikzlibrary{intersections}
\usepackage{pgfplots}
\usetikzlibrary{calc,3d,shapes, pgfplots.external, intersections}
\usepackage{pgfplots}
\usepackage{comment}

\usepackage[mathcal]{eucal}
\usepackage[all]{xy}

\usepackage{graphics,colortbl}

\input xy
\xyoption{all}
\usepackage{epsfig}

\newtheorem{theorem}{Theorem}[section]
\newtheorem{corollary}[theorem]{Corollary}
\newtheorem{lemma}[theorem]{Lemma}
\newtheorem{proposition}[theorem]{Proposition}
\theoremstyle{definition}
\newtheorem{definition}[theorem]{Definition}
\newtheorem{remark}[theorem]{Remark}

\newtheorem{problem}[theorem]{Problem}
\newtheorem{example}[theorem]{Example}
\numberwithin{equation}{section}
\usepackage{stmaryrd}
\usepackage{mathtools}
\usepackage{xcolor, soul}
\sethlcolor{green}

\pgfplotsset{compat=1.18} 
\title[A Dynamical Approach to Schur's Theorem]{A Dynamical Approach to Schur's Theorem}

\author[S. L'Innocente]{Sonia L'Innocente}
\address{Sonia L'Innocente \endgraf
School of Science and Technology -- University of Camerino\endgraf
via Madonna delle Carceri 9, Camerino, Italy\endgraf
Email: \texttt{sonia.linnocente@unicam.it} -- ORCID: \texttt{0000-0002-9224-7451}}

\author[F.G. Russo]{Francesco G. Russo}
\address{Francesco G. Russo \endgraf
School of Science and Technology -- University of Camerino\endgraf
via Madonna delle Carceri 9, Camerino, Italy\endgraf
and\endgraf
Department of Mathematics and Applied Mathematics -- University of the Western Cape\endgraf
Private Bag X17, 7535,  Bellville, South Africa\endgraf
and\endgraf
Department of Mathematics and Applied Mathematics -- University of Cape Town\endgraf
Private Bag X1, Rondebosch 7701, Cape Town, South Africa\endgraf
Email: \texttt{francesco.russo@unicam.it} -- ORCID: \texttt{0000-0002-5889-783X}}

\author[I. Svampa]{Ilaria Svampa}
\address{Ilaria Svampa\endgraf
Department Mathematik/Informatik--Abteilung Informatik -- Universit\"at zu K\"oln \endgraf
Albertus-Magnus-Platz, 50923, K\"oln, Germany\endgraf
Email: \texttt{ilaria.svampa@uni-koeln.de} -- ORCID: \texttt{0000-0002-1389-0319}
}

\begin{document}

\begin{abstract} A classical result of Schur of 1904 shows that an abstract group with finite central quotient has finite derived subgroup. Schur's Theorem has many important consequences and generalizations, which have been extensively investigated in the literature. We develop a new dynamical interpretation of Schur's Theorem for  locally compact  groups, using the notion of  topological entropy of Adler, Konheim and McAndrew.  We first consider groups with compact central quotient, introduced and called $\mathsf{Z}$-groups by Grosser and Moskowitz in the 1960s,
proving that if $G$ is a connected group such that $G/Z(G)$ is compact and with  continuous endomorphisms of finite 
topological entropy, then also $\overline{[G,G]}$   is compact and with continuous endomorphisms of finite topological entropy. The connectedness assumption is essential, since its absence allows us to construct a profinite group as counterexample. Furthermore, we study the Heisenberg groups $\mathbb{H}_n(R)$ on certain locally compact rings $R$ as a framework in which a dynamical Schur-Type Theorem persists, even though the central quotient need not be compact. In particular, we  find new formulas for the $p$-rank  of these Heisenberg groups.

\end{abstract}

\maketitle

\section{Introduction and  statement of the main results} 

Following \cite{HHR, heyer, hofmor},  a topological group is a group endowed with a topology for which the multiplication and inversion maps are continuous. Throughout the paper, all topological groups are assumed to be Hausdorff. In particular, compact and locally compact groups are understood in this sense. Discrete groups provide basic examples, and finite groups can be viewed as compact groups equipped with the discrete topology. For a group $G$, \begin{equation}\label{centerdefinition}Z(G)\coloneqq\{g \in G \mid [g,x]=1  \ \ \forall x \in G\}\end{equation} denotes the \textit{center of} $G$, while the \textit{derived subgroup of} $G$ is defined by
\begin{equation}\label{commutatordefinition}
[G,G]\coloneqq\langle [g,x]\coloneqq g^{-1}x^{-1}gx  \ \ \mid  \ \ x,g \in G \rangle.
\end{equation}

A classical theorem of Schur \cite{schur1} %(see also \cite[10.1.4]{rob}) 
states that if $G$ is a discrete (infinite) group such that $G/Z(G)$ is finite, then $[G,G]$ is finite. This result has played a fundamental role in group theory and has motivated a wide range of generalizations. In the framework of locally compact groups, the finiteness conditions are no longer appropriate as formulated for infinite discrete groups. A natural approach is therefore to replace the finiteness conditions by suitable dynamical properties which involve the notion of topological entropy. We adopt this perspective by considering the topological entropy of continuous endomorphisms, introduced by Adler, Konheim and McAndrew \cite{AdlerKonheimMcAndrew} and building on earlier ideas of Kolmogorov \cite{Kolmogorov} and Sinai \cite{Sinai}.  It has since been extended and applied in various contexts, including locally compact groups (see, for instance, \cite{B, GiordanoBrunoDikran, daf,  GBV, H, LW, RussoWaka}). In this setting, the topological entropy provides a quantitative tool to study the behavior of continuous endomorphisms and their interaction with the topological structure of the group.

Throughout this paper, we use Hood's uniform extension of the
Adler--Konheim--McAndrew topological entropy \cite{AdlerKonheimMcAndrew, B, H}, applied to the canonical
left uniformity of a locally compact group. More precisely, consider a locally compact group $G$, with $\mu$ left Haar measure on $G$, $\mathcal{CT}(G)$  collection of all the compact neighborhoods of the identity of $G$ and $\mathrm{End}(G)$ monoid of all the continuous endomorphisms of $G$. Note that the invertible elements of $\mathrm{End}(G)$ form the group of continuous automorphisms $\mathrm{Aut}(G)$. If $n\in\mathbb{N}$, $V\in\mathcal{CT}(G), \varphi\in \mathrm{End}(G)$ and 
$\varphi^{-n}(V)\coloneqq   \underbrace{\varphi^{-1} \left(\varphi^{-1}\dots( \varphi^{-1}(V))\dots\right)}_{n\mbox{-times}}$, we define the \emph{$n$-th $\varphi$-cotrajectory of $V$} as the set
$C_n(\varphi,V) \coloneqq V\cap\varphi^{-1}(V)\cap\ldots\cap\varphi^{-n+1}(V)$ which belongs to $\mathcal{CT}(G).$
The \emph{topological entropy of $\varphi\in \mathrm{End}(G)$ with respect to $V\in\mathcal{CT}(G)$} is defined by
\begin{equation}\label{eq:topentropPhiV}
\mathsf{H}_{\mathsf{top}}(\varphi,V)\coloneqq\limsup_{n\to\infty} \frac{-\log\mu(C_n(\varphi,V))}{n},
\end{equation}
and the \emph{topological entropy of $\varphi\in \mathrm{End}(G)$} is defined by
\begin{equation}\label{eq:topentropPhi}
\mathsf{h}_{\mathsf{top}}(\varphi)\coloneqq\sup\left\{\mathsf{H}_{\mathsf{top}}(\varphi,V)   \ \mid \ V\in\mathcal{CT}(G)\right\}.
\end{equation} 
The \emph{topological entropy of $G$} is defined by
\begin{equation}\label{importantset}
\mathsf{E}_{\mathsf{top}}(G)\coloneqq\{\mathsf{h}_{\mathsf{top}}(\varphi) \ \mid \  \varphi\in \mathrm{End}(G)\}.
\end{equation}
We investigate the cardinality of \eqref{importantset}, as done in \cite{LW, RussoWaka,   Y}.  Following  \cite{AdlerKonheimMcAndrew, GBV, RussoWaka, RussoWakabis, olwethuadditional}, we may introduce
\begin{equation}\label{classes}
\mathfrak{E}_0\coloneqq\{G \ \mbox{locally compact group}\ \mid \  \mathsf{E}_{\mathsf{top}}(G)=\{0\}\},\end{equation} \[\mathfrak{E}_{<\infty}\coloneqq \{G  \ \mbox{locally compact group} \ \mid \ \mathsf{E}_{\mathsf{top}}(G ) \subseteq [0, +\infty[
\}\]
and of course, $\mathfrak{E}_0 \subseteq \mathfrak{E}_{<\infty}$. Moreover finite groups and discrete groups belong to $\mathfrak{E}_0$. 

The class $\mathfrak{E}_{<\infty}$ turns out to be a significant replacement of the finiteness condition in Schur's Theorem \cite{schur1}. Locally compact groups with compact central quotient have been extensively studied; it is known from results of Grosser and Moskowitz \cite{gm1} that for a locally compact group $G$ if $G/Z(G)$ is compact, then also $\overline{[G,G]}$ is compact. This can be regarded as a topological version of Schur's Theorem and will be emphasized in Lemmas \ref{examplesofsingroups} and \ref{theor:structTh:Z-group} later on. Note that in a locally compact group $G$ the \textit{identity component} $G_0$ of $G$ is a closed fully invariant subgroup, which turns out to be the largest closed connected normal subgroup of $G$ containing the identity element \cite[Exercise E1.12]{hofmor}. In particular, a locally compact group $G$ is connected if and only if $G=G_0$, while $G/G_0$ is a totally disconnected locally compact group; a locally compact group $G$ is totally disconnected if and only if $G_0=1$. Our first main result shows that, under the same hypothesis, the finiteness of the topological entropy of continuous endomorphisms is inherited from the central quotient to the derived subgroup in case of connected groups. 

\begin{theorem}[Dynamical Formulation of Schur's Theorem]\label{1stmain} 
Let $G$ be a locally compact group such that $G/Z(G)$ is a compact group. If $G=G_0$ and $G/Z(G) \in \mathfrak{E}_{<\infty}$, then $\overline{[G,G]}\in %\mathfrak{E}_0\subseteq
\mathfrak{E}_{<\infty}$. 
On the other hand, if $G\neq G_0$ then the
conclusion does not hold in general: there exists a centerfree profinite group $H$ such that $H/Z(H)%=H\in \mathfrak{E}_0\subseteq
\in \mathfrak{E}_{<\infty}$ and $\overline{[H,H]}\not\in \mathfrak{E}_{<\infty}$.
\end{theorem}

The novelty of this result lies in the transfer of conditions of finiteness of the topological entropy of continuous endomorphisms from $G/Z(G)$ to $\overline{[G,G]}$, while the compactness of $\overline{[G,G]}$ under the compactness of $G/Z(G)$ is classical. The connectedness hypothesis is not a technical artifact: if $G$ is not connected, the conclusion of Theorem \ref{1stmain} does not hold in general, since a counterexample appears for profinite groups (i.e. compact groups which are totally disconnected). The obstruction in the profinite example is also conceptually important: endomorphisms of a fully invariant subgroup need not extend to endomorphisms of the ambient group. Consequently, entropy bounds on the ambient group do not automatically control all endomorphisms of its commutator subgroup.

Our second main result concerns Heisenberg groups $\mathbb{H}_n(R)$ over a concrete family of locally compact rings $R$ whose additive group $R^+$ is a locally compact abelian $p$-group of finite $p$-rank. In this setting, the central quotient need not be compact, but one can still obtain finiteness properties of the topological entropy for both $\mathbb{H}_n(R)/Z(\mathbb{H}_n(R))$ and $\overline{[\mathbb{H}_n(R),\mathbb{H}_n(R)]}$.

\begin{theorem}\label{2ndmain}
Let $p$ be a prime number, $n$ a positive integer and $\varepsilon,\zeta$ nonnegative integers such that $\varepsilon+\zeta>0$. Let $R\coloneqq \mathbb{Q}_p^\varepsilon \times \mathbb{Z}_p^\zeta$ be the ring obtained by the product of $\varepsilon$ copies of the $p$-adic rationals $\mathbb{Q}_p$ and by $\zeta$ copies of the $p$-adic integers $\mathbb{Z}_p$. Then the Heisenberg group $\mathbb{H}_{n}(R)$ is a periodic locally compact $p$-group of nilpotency class two such that  
\begin{equation}\label{eq:rankHnR}
\mathrm{rank}_p\bigl(\mathbb{H}_{n}(R)\bigr)  = (2n+1)\,\mathrm{rank}_p(R^+)\quad \textup{and}\quad \mathrm{rank}_p\bigl(\mathbb{H}_{n}(R)/Z(\mathbb{H}_{n}(R))\bigr)=2n\,\mathrm{rank}_p(R^+),
\end{equation} 
where $\mathrm{rank}_p(R^+)=\varepsilon+\zeta$.
In particular, both $\mathbb{H}_{n}(R)/Z(\mathbb{H}_{n}(R))  \in \mathfrak{E}_{<\infty}$ and $\overline{[\mathbb{H}_{n}(R),\mathbb{H}_{n}(R)]} \in  \mathfrak{E}_{<\infty}$.
\end{theorem} 

Notations and terminology (especially those on the concept of  \textit{finite $p$-rank} and  \textit{periodic locally compact groups}) are standard  and follow \cite{HHR, heyer, hofmor, karp, rob}. Moreover, they are recalled in Sections \ref{sec:prelimentropia}, \ref{sec:prelimZgruppi} and \ref{sez:prelimpgrouprankdim} later on. After a series of preliminary results on the notion of topological entropy of  locally compact groups in Section \ref{sec:prelimentropia}, we report some classical results on the representation theory of locally compact groups in Section \ref{sec:prelimZgruppi}, and the most used concepts of rank and dimension in Section \ref{sez:prelimpgrouprankdim} for the topological groups which we are going to investigate. Then we prove our first main result in Section \ref{sez:ourTHEOR1.1}, while our second main theorem is proved in Section \ref{sez:ourTHEOR1.2}.

%%%%%%%%%%%%%%%%%%%%%%%%%%%%%%%%%%%%%%%%%%%%%%
\section{Preliminary results of  topological entropy of locally compact groups}\label{sec:prelimentropia}

Let $G$ be a locally compact group. The topological entropy $\mathsf{h}_{\mathsf{top}}(\varphi)$ of a continuous endomorphism $\varphi$ of $G$ is defined as in Eq. \eqref{eq:topentropPhi}. We say that $V$ is $\varphi$-\textit{invariant} if $\varphi(V)\subseteq V$, i.e. if $\varphi^{-1}(V)\supseteq V$, and that $V$ is \emph{fully invariant} if it is $\varphi$-invariant for all continuous endomorphisms $\varphi$ in $\mathrm{End}(G)$. Then, $C_n(\varphi,V)\subseteq V$ captures how far $V$ is from being invariant under the first $n-1$ iterations of $\varphi^{-1}$. The less $V$ is $\varphi$-invariant, the larger $-\log\mu(C_n(\varphi,V))$ is, and vice versa.  As in many other contexts, the notion of topological entropy quantifies dynamical complexity and chaos: here $\mathsf{h}_{\mathsf{top}}(\varphi)$ reflects the degree to which $\varphi$  moves elements of $G$ and its subsets $V\in \mathcal{CT}(G)$. The topological entropy $\mathsf{E}_{\mathsf{top}}(G)$ of $G$, defined as in Eq. \eqref{importantset}, is the spectrum of values of the topological entropy of all the continuous endomorphisms on $G$. 

Throughout Sections \ref{sec:prelimentropia}--\ref{sez:ourTHEOR1.1}, let $\mathbb{Q}_p$ and
$\mathbb{Z}_p$ denote respectively the additive topological group of $p$-adic rationals and $p$-adic integers (see \cite[Exercise E1.16]{hofmor}), endowed with their usual $p$-adic topologies. Their multiplicative ring structures will be invoked explicitly only in Section \ref{sez:ourTHEOR1.2} for Theorem \ref{2ndmain}. Moreover, let $\mathbb{Q}$ and $\mathbb{R}$ denote respectively the additive topological group of rationals and reals, with standard Euclidean topology. A formula of  Yuzvinski \cite{Y} shows that $\mathsf{h}_{\mathsf{top}}(\varphi)$ can be calculated from the solutions of the characteristic polynomial of $\varphi$ when $\varphi\in\mathrm{End}(\mathbb{R}^{m})$ or $\varphi\in\mathrm{End}(\mathbb{Q}_p^{m})$ for any $m \in\mathbb{N}$, see \cite{daf, LW}. Note also that  $\mathbb{Q}_p$ is a totally disconnected locally compact abelian group, but $\mathbb{Q}$ is not locally compact when endowed with the induced topology from $\mathbb{R}$.  

\begin{example}\label{ex:RQpentropiafin}
Yuzvinski's Formula \cite{B,LW,Y} allows us to calculate the topological entropy of $\mathbb{R}^{m}$ and $\mathbb{Q}_p^{m}$, for every prime $p$ and every $m\in\mathbb{N}$. For $\varphi\in\mathrm{End}(\mathbb{R}^{m})$, one has 
\begin{equation}\label{eq:entropRlog}
 \mathsf{h}_{\mathsf{top}}(\varphi)=\underset{\vert\lambda_{i}\vert>1}{\sum}\log\vert\lambda_{i}\vert,   
\end{equation}
where $\left\{ \lambda_{i} \mid i=1,2,\ldots,m\right\}$ is the family of eigenvalues of $\varphi$. Similarly for $\mathbb{Q}_p^{m}$, denoting the $p$-adic norm by $|\cdot|_p$, one can show that
\begin{equation}
\mathsf{h}_{\mathsf{top}}(\varphi)=\sum_{|\lambda_i|_p>1}\log|\lambda_i|_p,
\end{equation}
where $\left\{ \lambda_{i} \mid i=1,2,\ldots,m\right\}$ is the set of eigenvalues of $\varphi$ in a finite extension of $\mathbb{Q}_p$. In particular, 
\begin{equation}
    \mathbb{R}^{m},\ \mathbb{Q}_p^m\in\mathfrak{E}_{<\infty}\setminus\mathfrak{E}_{0}.
\end{equation}
\end{example}

The characterization of groups in $\mathfrak E_0$ can indicate the presence of structural theorems. This was largely discussed in \cite{AdlerKonheimMcAndrew, B, H, LW,  RussoWaka, RussoWakabis, olwethuadditional}.

\begin{remark} If $G$ is a discrete locally compact group, then its Haar measure $\mu$  becomes the counting measure on $G$ and we find $G\in \mathfrak{E}_0$, see \cite{AdlerKonheimMcAndrew, B, H}. \end{remark}

The finiteness of the topological entropy is less known for large classes of nonabelian locally compact groups, see \cite{RussoWaka, RussoWakabis}; however in the abelian case there are results which  allow us to find conditions under which a group belongs to $\mathfrak E_{<\infty}$ or $\mathfrak E_0$.

\begin{example}[See \cite{olwethuadditional}, Some Totally Disconnected Locally Compact Abelian Groups] \label{lcapg}
\
    \begin{itemize}
        \item[1.] The Pr\"ufer group $\mathbb{Z}(p^{\infty})$ is  discrete abelian in $\mathfrak E_0$.
        \item[2.] The additive group of $p$-adic integers $\mathbb{Z}_{p}$ is compact abelian in $\mathfrak E_0$. Moreover, it has the structure of  $p$-adic ring and is an example of compact ring.
        \item[3.] The additive group of $p$-adic rationals $\mathbb{Q}_{p}$ is  a noncompact locally compact abelian group in $\mathfrak E_{<\infty}$ but not in $\mathfrak E_0$.  It is also a $p$-adic field and in fact it is an example of locally compact field.
    \end{itemize}
\end{example}

\begin{remark}\label{rem:vanDantzig}
We mention some useful results which simplify the calculation of the topological entropy in the case of totally disconnected groups. For a totally disconnected locally compact group $G$, van Dantzig \cite{vanDantzig} proved that the family
\begin{equation}
    \mathcal{B}(G)\coloneqq\{U\leq G \mid U\textup{ compact and open}\}\subseteq \mathcal{CT}(G)
\end{equation}
is a local basis of neighborhoods of the identity of $G$. From \cite[Lemma 3.6]{GBV}, if $G$ is a totally disconnected locally compact group, $\varphi\in\mathrm{End}(G)$ and $U,V\in \mathcal{B}(G)$, $U \leq V$, then $\mathsf{H}_{\mathsf{top}}(\varphi,V) \leq \mathsf{H}_{\mathsf{top}}(\varphi,U)$. As a consequence, when computing the topological entropy of endomorphisms on totally disconnected locally compact groups, the supremum in  \eqref{eq:topentropPhi} can be restricted to the elements of $\mathcal{B}(G)$ only,
\begin{equation}\label{eq:topentropPhitd}
\mathsf{h}_{\mathsf{top}}(\varphi)=\sup\left\{\mathsf{H}_{\mathsf{top}}(\varphi,V)  \  \ \mid \ \ V\in\mathcal{B}(G)\right\},
\end{equation}
moreover \eqref{eq:topentropPhiV} can be simplified for $ V\in\mathcal{B}(G)$, $\varphi\in\mathrm{End}(G)$ as
\begin{equation}
\mathsf{H}_{\mathsf{top}}(\varphi,V)=\lim_{n\to\infty} \frac{\log[V:C_n(\varphi,V)]}{n}.
\end{equation}
\end{remark}

The following example is discussed also in \cite{RussoWaka}.
\begin{example}\label{ex:Zpentropia0}
For every prime $p$ and every $m\in\mathbb{N}$, the group $\mathbb{Z}_p^m$ has a local basis of neighborhoods of $\mathbf{0}$  -- e.g. $\left\{p^k\mathbb{Z}_p^m\mid k\in\mathbb{N}\cup\{0\}\right\}$ 
--- consisting of fully invariant subgroups. Therefore $\mathsf{h}_{\mathsf{top}}(\varphi)=0$ for every $\varphi\in\mathrm{End}(\mathbb{Z}_p^m)$, and $\mathbb{Z}_p^m\in \mathfrak{E}_0$. 
\end{example}

We shall also use the entropy of Bernoulli shifts, which we recall in the following example. 
\begin{example}[Bernoulli shift, see \cite{DikranSanchis}, Example 5.1]\label{ex:Bernoulli}
If $K$ is a compact group, the (\emph{left}) \emph{Bernoulli shift} is
\begin{equation}
\varphi_K\colon K^\mathbb{N}\rightarrow K^\mathbb{N},
 \qquad
 \varphi_K(x_1,x_2,x_3,\ldots)\coloneqq(x_2,x_3,\ldots).
\end{equation}
Its entropy is
\begin{equation}
 \mathsf{h}_{\mathsf{top}}(\varphi_K)=\log|K|,
\end{equation}
which is finite when $K$ is finite, while it is intended to be $\infty$ for an infinite group $K$.
\end{example}

We end with some \emph{Addition Theorems} which balance the finiteness of the topological entropy. 
\begin{lemma}[Addition Theorems]\label{teorspezzamento}
Let $G$ be a locally compact group, $\varphi\in \mathrm{End}(G)$, $N$ a $\varphi$-invariant closed normal subgroup of $G$, and  $\varphi_{G/N}\in \mathrm{End}(G/N)$ the endomorphism induced by $\varphi$ as $\varphi_{G/N}(gN)\coloneqq \varphi(g)N$ for all $g\in G$.
\begin{itemize}
\item[{\rm (i).}] Then \cite[Lemma 2.3(a)]{daf} (see also \cite{AdlerKonheimMcAndrew}, \cite{H}) implies
\[
\mathsf{h}_{\mathsf{top}}(\left.\varphi\right\rvert_N) \le \mathsf{h}_{\mathsf{top}}(\varphi) \qquad \textup{and}\qquad \mathsf{h}_{\mathsf{top}}(\varphi_{G/N})\le \mathsf{h}_{\mathsf{top}}(\varphi).
\]
\item[{\rm (ii).}] If $G$ is compact, then \cite[Theorem 4.3.4]{GiordanoBrunoDikran} (see also \cite{Y,B,GBV}) implies
\[
\mathsf{h}_{\mathsf{top}}(\varphi)= \mathsf{h}_{\mathsf{top}}(\left.\varphi\right\rvert_N)+\mathsf{h}_{\mathsf{top}}(\varphi_{G/N}).
\]
\item[{\rm (iii).}]
If $G$ is totally disconnected, $\varphi(N)=N$ and $\ker\varphi\subseteq N$, then \cite[Theorem 1.2]{GBV} implies
\[
\mathsf{h}_{\mathsf{top}}(\varphi)= \mathsf{h}_{\mathsf{top}}(\left.\varphi\right\rvert_N)+\mathsf{h}_{\mathsf{top}}(\varphi_{G/N}).
\]
\item[{\rm (iv).}] If $G$ is compact and $N$ is fully invariant such that $N,\,G/N\in\mathfrak{E}_{<\infty}$, then $G\in\mathfrak{E}_{<\infty}$. The same conclusion holds even if we replace $\mathfrak{E}_{<\infty}$ with $\mathfrak{E}_0$. \end{itemize}
\end{lemma}
The role of Lie groups is quite special among locally compact groups, as pointed out by Patr\~{a}o \cite{patrao1, patrao2} using the theory of the Li-Yorke pairs \cite[Section 5]{patrao2} (note that Patr\~{a}o uses a different extension of topological entropy to noncompact spaces%defined in terms of admissible open covers and characterized by a variational principle
, which coincides with Hood's topological entropy used here on compact groups).

%%%%%%%%%%%%%%%%%%%%%%%%%%%%%%%%%%%%%%%%%%%%%%%%%%
\section{Preliminary results of representation theory of locally compact groups
}\label{sec:prelimZgruppi}
We recall the notion of central groups, introduced and systematically studied by Grosser and Moskowitz.  This class provides a natural common extension of compact groups and locally compact abelian groups, while retaining strong compactness properties for the nonabelian part of the group.

\begin{definition}[See \cite{gm1}]
A locally compact group $G$ is a $\mathsf{Z}$-\emph{group} (or \textit{central group}) if $G/Z(G)$ is compact.
\end{definition}

We also recall other notions to give context and to locate $\mathsf{Z}$-group within a classical hierarchy of topological groups.

\begin{definition}[See \cite{heyer,palmer,hofmann2,gm1}] \label{relevantnotions}
Let $G$ be a locally compact group. 
\begin{itemize}
    \item [{\rm (i).}] $G$ is a $\mathsf{MAP}$-\emph{group} (or \textit{maximally almost periodic group}) if $G$ has a continuous injective homomorphism into a compact group. 
    \item [{\rm (ii).}] $G$ is a $\mathsf{T}$-\emph{group} (or \textit{Takahashi group}, see \cite{takahashi}) if $G$ is a $\mathsf{MAP}$-group  and $\overline{[G,G]}$ is compact.
    \item [{\rm (iii).}] $G$ is a \emph{Moore group} if all its irreducible representations are finite-dimensional.
    \item[{\rm (iv).}] $G$ is a $\mathsf{SIN}$-\emph{group} (or $G$ is said to have \emph{small invariant neighborhoods}, see \cite{hofmor}) if $G$ has a local basis of neighborhoods of the identity which are invariant under all inner automorphisms of $G$.
   \end{itemize}
Moreover, following the convention in \cite[p. 64, (1)]{hofmann2}, a topological group $G$ is a \emph{pro-Lie group} if for every identity neighborhood $U$ there is a compact normal subgroup $N \subseteq U$ such that $G/N$ is a finite-dimensional Lie group.
\end{definition}
It is convenient to look at the definition of pro-Lie group introducing the filter basis
\begin{equation}
    \mathcal{N}(G)\coloneqq\{N \ \mbox{is a closed normal subgroup of} \ G \mid G/N \ \mbox{is a finite-dimensional Lie group}\},
\end{equation}
ordered by reverse inclusion, and noting that $G$ is a pro-Lie group if and only if $G$ is locally compact and $ \mathcal{N}(G)$ converges to the identity, the latter meaning that for every neighborhood $U_0$ of the identity  there is an $N \in \mathcal{N}(G)$ such that $N \subseteq U_0$. In particular, this is equivalent to saying that a locally compact group $G$ is a pro-Lie group if $G = \underset{N \in \mathcal{N}(G)}{\varprojlim} G/N$ is  projective limit of Lie groups.
These locally compact groups find a natural perspective  of study within the context of representation theory  (where representations are assumed to be continuous and unitary); see \cite{{heyer}}. A classical result which structurally describes locally compact groups is the following:

\begin{lemma}[See \cite{heyer},  Gelfand--Raikov Theorem, 
pp. 13--14]\label{gr}
Let $G$ be a locally compact group. The set of all irreducible representations of $G$ separates the points of $G$. That is, for every $g,h\in G$, there exists a complex Hilbert space $\mathcal{H}$ and a unitary irreducible representation $\rho\colon G\rightarrow\mathrm{U}(\mathcal{H})$ such that $\rho(g)\neq \rho(h)$.     
\end{lemma}

Another classical result which describes compact groups is due to their \textit{regular representation} on the Hilbert space $L^2(G,\mathbb{K})$ of square-integrable functions from $G$ to $\mathbb{K}$, where $\mathbb{K}$ is either the field of real numbers $\mathbb{R}$ or that of complex numbers $\mathbb{C}$, see \cite[Example 2.12]{hofmor}. This implies what is known as the Fundamental Theorem on Unitary Modules \cite[Theorem 2.22]{hofmor} and consequently a relevant structural result for compact groups:

\begin{proposition}[See \cite{hofmor}, Corollaries 2.27 and 2.29]\label{regrep} If $G$ is a compact group, then the finite-dimensional simple modules separate the points. In particular, every compact group $G$ is isomorphic to a closed subgroup of a product of orthogonal groups $\prod_{j \in J}\mathrm{O}(n_j)$ and of a product $\prod_{j \in J}\mathrm{U}(n_j)$ of unitary groups (with $n_j\in\mathbb{N}$).    
\end{proposition}

\begin{remark}[See \cite{heyer}]\label{rem:repnheyergr}
Let $G$ be a locally compact group. A function $f\in\mathscr{C}_{\mathbb{C}}^b(G)$ is called \emph{almost periodic} if one of the following equivalent conditions is satisfied:
\begin{itemize}
    \item[{\rm (i).}] $\{f\circ L(a) \mid a\in G\}$ is relatively compact in $\mathscr{C}_{\mathbb{C}}^b(G)$,    where $L(a)\colon g\mapsto ag$, $g\in G$ is the left translation on $G$;
    \item[{\rm (ii).}] $\{f\circ R(a)\mid a\in G\}$ is relatively compact in $\mathscr{C}_{\mathbb{C}}^b(G)$,    where $R(a)\colon g\mapsto ga$, $g\in G$ is the right translation on $G$.
\end{itemize}
Denoting by $\mathfrak{U}(G)$ the $C^\ast$-algebra of all almost periodic functions on $G$, one can see that $G$ is a $\mathsf{MAP}$-group if and only if $\mathfrak{U}(G)$ separates the points of $G$, or equivalently, the finite-dimensional unitary representations of $G$ separate the points of $G$ (see \cite[pp.~13--15]{heyer}). In particular, by Gelfand-Raikov Theorem \ref{gr}, every Moore group is a $\mathsf{MAP}$-group.
\end{remark}

We collect some facts on the locally compact groups introduced above, which offer interesting examples.

\begin{lemma}[See \cite{heyer}, pp. 13--16, \cite{gm1} and \cite{hofmann2} pp. 64--65] \label{examplesofsingroups} \
\begin{itemize}
\item [{\rm (i).}] If $G$ is a pro-Lie group, then $G$ is a locally compact group.
\item [{\rm (ii).}] If $G$ is a locally compact group such that $G/G_0$ is compact, then $G$ is a pro-Lie group. In particular, every locally compact group contains open pro-Lie subgroups.
\item [{\rm (iii).}] If $G$ is either a compact group or a locally compact abelian group, then $G$ is a $\mathsf{Z}$-group.
\item[{\rm(iv).}] If $G$ is a $\mathsf{Z}$-group, then $G$ is a $\mathsf{T}$-group.
\item [{\rm (v).}] If $G$ is a $\mathsf{T}$-group, then $G$ is a Moore group.
\item [{\rm (vi).}] If $G$ is a Moore group, then $G$ is a $\mathsf{MAP}$-group.
\item[{\rm (vii).}] If $G$ is a Moore group, then $G$ is a $\mathsf{SIN}$-group. 
\item[{\rm (viii).}] If $G$ is a $\mathsf{SIN}$-group, then $G$ is a pro-Lie group. 
\end{itemize}
\end{lemma}

We may visualize the inclusions of Lemma \ref{examplesofsingroups}(iii)--(viii) as follows:
\begin{align*}
&\{\textup{{\small compact groups}}\},\, \{\textup{{\small locally compact abelian groups}}\}\subset \{\mathsf{Z}\textup{-groups}\}\subset\{\mathsf{T}\textup{-groups}\}\subset \{\textup{Moore groups}\};\\
&\{\textup{Moore groups}\}\subset\{\mathsf{MAP}\textup{-groups}\};\\
& \{\textup{Moore groups}\}\subset\{\mathsf{SIN}\textup{-groups}\}\subset \{\textup{pro-Lie groups}\}.
\end{align*}
All the symbols of inclusion displayed above are proper, as one can provide examples where those inclusions are strict. Details can be found in \cite[Preliminaries and Chapter I]{heyer} and in \cite{palmer}. Neither the class of $\mathsf{MAP}$-groups nor the class of pro-Lie groups contains the other. It is very instructive to give a look at \cite[Chapter 0]{hofmann2}, in order to find examples of locally compact groups which are realized as semidirect products and are close or far from those which are listed in Lemma \ref{examplesofsingroups}.

\medskip

The structural result that will be used later is the following relevant condition of splitting for $\mathsf{Z}$-groups by Grosser and Moskowitz.
\begin{lemma}[See \cite{gm1}, Theorem~4.4, and Corollary~2 at p.~331]\label{theor:structTh:Z-group}
Let $G$ be a $\mathsf{Z}$-group. Then $G\simeq W\times H$, where $W\simeq \mathbb{R}^m$  for some $m\geq0$, and $H$ contains a compact open normal subgroup $K$ such that $H/K$ is abelian and discrete.
\end{lemma}

As a consequence of Lemma \ref{theor:structTh:Z-group},  \begin{equation}\label{commutatorcalculus}
\overline{[G,G]} = \overline{[\mathbb{R}^m\times H,\mathbb{R}^m\times H]} = \overline{[\mathbb{R}^m,\mathbb{R}^m]\times[H,H]} = \overline{[H,H]},\end{equation}  
which is compact as shown in \cite[Corollary 1, p. 331]{gm1}.

\medskip

We conclude this section with some results on compact connected groups.
\begin{definition}[See \cite{hofmor}, Definition 9.5]
A compact connected group $S$ is \emph{semisimple} if $S=[S,S]$.
\end{definition}
Note that if $S$ is a compact connected group then $[S,S]$ is closed \cite[Theorem 9.2]{hofmor}, hence compact and connected; moreover $[S,S]$ is semisimple \cite[Corollary 9.6]{hofmor}. Furthermore, if $S$ is a compact connected group, then the subgroup $[S,S]\cap Z(S)_0$ is totally disconnected central and $S=[S,S]Z(S)_0$ \cite[Theorem 9.24]{hofmor}; in particular, $Z(S)$ is totally disconnected when $S$ is compact connected semisimple \cite[Theorem 9.19 (ii)]{hofmor}.
\begin{lemma}[See \cite{hofmor}, Corollary 9.20, and \cite{DikranSanchis}, Fact 2.5, Remark 2.6; Lie Components]
\label{lem:lie-components} 
Let $S$ be a compact connected semisimple group. Then there exist
an at most countable index set $J$, a family $\{L_j\}_{j\in J}$ of pairwise non-isomorphic compact connected centerfree simple Lie
groups, and a family $\{\kappa_{j}\}_{j\in J}$ of nonzero cardinals such that
\begin{equation}
\label{eq:lie-component-decomposition}
S/Z(S)\simeq \prod_{j\in J} L_{j}^{\kappa_{j}}
\end{equation}
as topological groups, where $\prod$ denotes the cartesian product and $L_{j}^{\kappa_{j}}$ the cartesian power. The groups $L_{j}$ are called the \emph{Lie components} of
$S$, and $\kappa_{j}$ is called the \emph{multiplicity} of
the Lie component $L_j$ in $S$. The family of pairs $\bigl\{([L_{j}],\kappa_{j})\bigr\}_{j\in J}$ is uniquely determined by \(S\), up to a reindexing of $J$.
\end{lemma}

%%%%%%%%%%%%%%%%%%%%%%%%%%%%%%%%%%%%%%%%%%%%%
\section{Preliminary results on the rank and the dimension}\label{sez:prelimpgrouprankdim}
In this section $p$ always denotes a prime number.
\begin{definition}[See \cite{HHR}, Definition 2.2(ii), $p$-Groups]\label{Abelian_p_groups} For a locally compact group $G$, an element $g\in G$ is called $p$-\textit{element}, if the sequence $g^{p^{k}}$ with $k\in\mathbb{N}$ converges to the identity element in $G$. A locally compact group $G$ is called $p$-\textit{group} if 
$G$ coincides with \begin{equation} G_p\coloneqq\left\{ g\in G \mid g\text{ is  }p\text{-element}\right\}. \end{equation}
A maximal $p$-subgroup of a locally compact group $G$ is called \textit{Sylow $p$-subgroup} of $G$, or primary $p$-component of $G$. 
  \end{definition}
The groups $\mathbb{Z}_p$, $\mathbb{Q}_p$ and $\mathbb{Z}(p^\infty)$ mentioned in Example \ref{lcapg} are locally compact abelian $p$-groups. The class of locally compact $p$-groups strictly contains the class of pro-$p$ groups defined in \cite[Sec. 2.1]{rz}: a pro-$p$ group is a projective limit of finite $p$-groups (e.g. $\mathbb{Z}_p$ is a pro-$p$ group, while $\mathbb{Q}_p$ and $\mathbb{Z}(p^\infty)$ are not). Note also that 
if $G$ is a totally disconnected locally compact group then the set $G_p$ is closed by \cite[Lemma 2.6]{HHR}. 

From \cite[p. 5]{HHR}, a \textit{compact element} of $G$ is an element $g \in G$ such that $\overline{\langle g \rangle}$ is compact in $G$ and the set \begin{equation} \mathrm{comp}(G)\coloneqq \{g \in G \mid g \ \mbox{is  compact element}\} \end{equation} is described in \cite[Proposition 1.3, Lemma 1.6]{HHR}. For instance, $\mathrm{comp}(G)$ is  a subgroup of $G$ when $G$ is locally compact abelian, but in general $\mathrm{comp}(G)$ is just a subset of $G$, not necessarily a subgroup.
  
\begin{definition}[See \cite{HHR}, Definition 2, Periodic Groups]\label{plcg} 
A locally compact group $G$ is \emph{periodic} if it is totally disconnected (i.e. $G_0$ is trivial) and $\overline{\langle g \rangle}$ is compact for all $g \in G$. \end{definition}

The groups $\mathbb{Z}_p$, $\mathbb{Q}_p$ and $\mathbb{Z}(p^\infty)$ (see Example \ref{lcapg}) are periodic abelian groups. The considerations on Heisenberg groups over $\mathbb{Z}_p$ and $\mathbb{Q}_p$ in Theorem \ref{2ndmain} and in \cite{RussoWaka, RussoWakabis} show that there are  periodic nonabelian groups. Note that Definition \ref{Abelian_p_groups} of locally compact $p$-group is weaker than the one given in \cite[Definition 2.2(ii)]{HHR}, since the latter also requires the group to be periodic. However the two definitions agree within the class of locally compact abelian groups, see \cite[pp. 48--49]{HHR}.

\begin{remark}\label{Abelian_p_groups_bis} Definition \ref{Abelian_p_groups} for  a compact abelian group is equivalent to \cite[Definitions 8.7(i)]{hofmor}: a compact abelian group $G$ is called a \textit{compact $p$-group} if its  \textit{Pontryagin dual} $\widehat{G}$ is a $p$-group. A locally compact abelian group turns out to be a $p$-\textit{group} if it is a union of compact abelian $p$-groups.
\end{remark}
If $G$ is a compact abelian group, then  $\widehat{G}$ is a discrete abelian group \cite[Proposition 7.5(i)]{hofmor}, and a $p$-group may be defined up to its Pontryagin dual; the situation extends to 
locally compact abelian groups. At this point, we shall mention the delicate notion of $p$-rank.

\begin{definition}[See \cite{HHR}, Topologically Finitely Generated Groups]\label{tfgg} A locally compact group $H$ is {\em topologically finitely generated} if there exists a finite subset $Y$ of $H$ such that $H=\overline{\langle Y \rangle}$. For a  topologically finitely generated locally compact group $H$, define the \textit{minimal number of topological generators} of $H$ (also called \textit{defect of H}) by \begin{equation}\label{defect} d(H)\coloneqq\min \{|Y| \ | \ Y \subseteq H \ \mbox{and} \ \overline{\langle Y \rangle} = H\}. \end{equation} 
The  $p$-\textit{rank} of a locally compact $p$-group $G$ is \begin{equation}\label{eq:topological-p-rank} \mathrm{rank}_p(G) \coloneqq \sup \{ d(H) \ | \ H \ \mbox{topologically finitely generated closed subgroup of} \ G\}. \end{equation}
\end{definition}
As usual, we write $d(H)=\infty$ when $H$ is not
topologically finitely generated. We may rephrase Definition \ref{tfgg} saying that $\mathrm{rank}_p(G) \le r$ for some $r\in\mathbb{N}$ means that there are no topologically finitely generated closed subgroups in $G$ that can be generated by more than $r$ elements; if such $r$ does not exist, we say that $G$ does not have finite $p$-rank and write $\mathrm{rank}_p(G) =\infty$. For example, $\mathrm{rank}_p(\mathbb{Q}_p^\varepsilon)= \varepsilon$, $d(\mathbb{Q}_p^\varepsilon)= \infty$ and $\mathrm{rank}_p(\mathbb{Z}_p^\zeta) = d(\mathbb{Z}_p^\zeta)= \zeta$, for $\varepsilon,\zeta\in\mathbb{N}$. For a discrete abelian $p$-group, \eqref{eq:topological-p-rank} agrees with the usual notion  of $p$-\textit{rank} in \cite[Chapter 4, \S 4.2, p.~99]{rob}, also called \textit{Pr\"ufer} $p$-\textit{rank of an abelian group}. For a topologically finitely generated profinite group, Definition 2.4.5 in \cite{rz} agrees with \eqref{defect} here. Just to give an idea of why restrictions of structural nature arise from finiteness of the $p$-rank or from small enough sets of topological generators, we recall  that a locally compact group $G$ is \textit{compactly generated} if it has a compact subset $C$ generating $G$, that is, $\langle C \rangle=G$, see  \cite[Proposition 5.75]{hofmor}.

 \begin{proposition}[See \cite{hofmor}, Theorem 7.57(ii)]\label{CompactlyGenerated} Every compactly generated locally compact abelian group is isomorphic to the direct sum $\mathbb{R}^d \oplus \mathbb{Z}^m \oplus K$ for a compact abelian group $K$ and two nonnegative integers $d, m$. \end{proposition} 
 
Their topological entropy has been described in \cite{RussoWaka, RussoWakabis}. The following result involves locally compact abelian $p$-groups of finite $p$-rank, where the notion of finite $p$-rank brings further decompositions.

\begin{lemma}[See \cite{HHR}, Theorem 3.97]\label{p-rank-finite}
For a locally compact abelian $p$-group $G$, we have  
\begin{equation}\label{eq:locompabpgrfinrank}
\mathrm{rank}_p(G)<\infty \ \ \Longleftrightarrow \ \ G \simeq \mathbb{Q}_p^\alpha \times\mathbb{Z}_p^\beta \times \mathbb{Z}(p^\infty)^\gamma \times E_p
\end{equation}
for a finite abelian $p$-group $E_p$ of $\mathrm{rank}_p(E_p)=\delta$ and for some $\alpha, \beta, \gamma, \delta\in\mathbb{N}\cup\{0\}$. In particular, \begin{equation}\mathrm{rank}_p(G)=\alpha + \beta + \gamma + \delta.\end{equation}
\end{lemma}

Also for these groups the finiteness of the topological entropy has been studied.

\begin{lemma}[See \cite{daf, RussoWaka}, Small Entropy vs Small $p$-rank] \label{prankfin-entropy}
If $G$ is a locally compact abelian $p$-group with finite $p$-rank, then $G \in \mathfrak{E}_{<\infty}$, and $G\in \mathfrak{E}_0$ if and only if $\alpha=0$ in  \eqref{eq:locompabpgrfinrank}.
\end{lemma}

To prove these results, and in other calculations, the main idea is to look for decompositions of the endomorphisms in portions of the group where we can control the finiteness of the topological entropy, so to be able to apply some Addition Theorems (see, e.g., Lemma \ref{teorspezzamento}).

\medskip

Here we shall also mention the notion of \textit{(topological) dimension}, that we will use for locally compact groups. For a locally compact Hausdorff space $X$, we denote by $\dim(X)$ the Lebesgue covering dimension of $X$. A locally compact group is totally disconnected if and only if it is zero-dimensional \cite[Exercise E8.6]{hofmor}. Hofmann and Morris place the notion of dimension in a broader axiomatic framework \cite[Scholium~8.25]{hofmor}, and show that this agrees with the standard notions of topological dimension on locally compact abelian groups.

\begin{remark}[See \cite{hofmor}, Dimension of Locally Compact Groups] \label{dimensions} Consider the following function 
\begin{equation}\mathrm{DIM} \colon X \in \mathcal{C} \longmapsto \mathrm{DIM}(X) \in \{0\}\cup\mathbb{N} \cup \{+\infty\}\end{equation} defined on the class $\mathcal{C}$ of all locally compact topological spaces. Assume that the following conditions are satisfied:
\begin{itemize}
    \item [{\rm (Da).}] If $f : X\in \mathcal{C} \to Y\in \mathcal{C}$ is a covering map, then $\mathrm{DIM} (X) = \mathrm{DIM} (Y)$.
    \item [{\rm (Db)}.] If $X=\mathbb{R}^n$ or $X=[0,1]^n$ for $n \in \mathbb{N}$, then $ \mathrm{DIM} (X)=n$.
    \item [{\rm (Dc)}.] For every paracompact space $Y \in \mathcal{C}$ and each closed subspace $X$ of $Y$ we have $\mathrm{DIM} (X) \le \mathrm{DIM} (Y)$.
    \item [{\rm (Dd)}.] Assume that $X$ is the underlying space of a compact group whose topology has a basis of compact open sets and assume that $Y=\mathbb{R}^n$. Then  $\mathrm{DIM} (X \times Y) \le n$.
    \end{itemize}
Then for every compact abelian group $G$ we have that 
    \begin{equation}\mathrm{DIM}(G)=\mathrm{dim}(G)=\mathrm{dim}_\mathbb{Q}(\mathbb{Q} \otimes_{\mathbb{Z}} \widehat{G}),\end{equation}
where  $\mathrm{dim}_\mathbb{Q}(\mathbb{Q} \otimes_{\mathbb{Z}} \widehat{G})$ denotes the dimension as $\mathbb{Q}$-module of the tensor product  $\mathbb{Q} \otimes_{\mathbb{Z}} \widehat{G}$.\\
For any locally compact abelian group $G=\mathbb{R}^n \times H$ with $H$ compact abelian group, 
\begin{equation}\label{observationeasy} \mathrm{DIM}(G)=n + \mathrm{dim}(\mathrm{comp}(G)_0).\end{equation}
\end{remark}

On this basis, we may recall a relevant result which describes the finiteness of the topological entropy in case of locally compact abelian groups. 

\begin{proposition}[See \cite{daf}, Theorem 1.1; see also \cite{DikranSanchis, GBV, LW, RussoWaka}, Influence  of the Finite Dimension on the Topological Entropy]\label{origin}
Let $G$ be a locally compact abelian group. 
\begin{itemize}
\item[{\rm (i)}.] If $G\in\mathfrak{E}_{<\infty}$, then $\mathrm{dim}(G) < +\infty$.  
\item[{\rm (ii)}.] The converse of {\rm (i)} above is true when $G$ is compact and $G/G_0\in \mathfrak{E}_{<\infty}$.
\item[{\rm (iii)}.]  If $G\in \mathfrak{E}_0$, then  $\mathrm{dim}(G) =0$.
\item[{\rm (iv)}.] If $G$ is profinite, then $G\in \mathfrak{E}_0$ if and only if $G\in \mathfrak{E}_{<\infty}$.
\end{itemize}
\end{proposition}

%%%%%%%%%%%%%%%%%%%%%%%%%%%%%%%%%%%%%%%%%%%%%%%%%%%%%%%%%%
\section{The first main result and its proof}\label{sez:ourTHEOR1.1}

Several extensions of Schur's Theorem have been investigated in the literature, both in the setting of abstract groups and in topological group theory, see \cite{hofmann1, leptin1, leptin2, niroomand, podoski}. 
We now prove the main result of this section.

\begin{proof}[Proof of Theorem \ref{1stmain}]  Let $G=G_0$ be a $\mathsf{Z}$-group such that $G/Z(G)\in \mathfrak{E}_{<\infty}$. By \cite[Corollary~1, p.~331]{gm1}, or equivalently from Lemma \ref{examplesofsingroups}(iv), $\overline{[G,G]}$ is a compact group. By Lemma \ref{theor:structTh:Z-group}, we have that $G \simeq \mathbb{R}^m \times H$ for some $m \ge 0$ and for some subgroup $H$ containing a compact open normal subgroup $K$ such that $H/K$ is discrete abelian. Since $G$ is connected, so is $H$, with $K$ open in $H$, hence $H=K$. Therefore $G \simeq \mathbb{R}^m \times K$ where $K$ is compact connected and we conclude from \eqref{commutatorcalculus} that $\overline{[G,G]}= \overline{[H,H]}= \overline{[K,K]}$. We have $\overline{[K,K]}=[K,K]$ compact connected and semisimple by \cite[Corollary 9.6]{hofmor}, with
\begin{equation} \label{eq:truefirsteq}
\frac{G}{Z(G)} = \frac{\mathbb{R}^m \times K}{Z(\mathbb{R}^m \times K)} \simeq  \frac{\mathbb{R}^m }{Z(\mathbb{R}^m)} \times \frac{K}{Z(K)} \simeq   \frac{K}{Z(K)} \in \mathfrak{E}_{<\infty}.\end{equation}
By \cite[Theorem 9.24]{hofmor}, since $Z(K)_0\leq Z(K)$, we deduce that $K=[K,K]Z(K)$. Therefore 
\begin{equation}\frac{K}{Z(K)}\simeq \frac{[K,K]}{[K,K]\cap Z(K)}.\end{equation} Moreover $[K,K]\cap Z(K) = Z([K,K])$, thus 
\begin{equation}
\frac{K}{Z(K)}\simeq \frac{[K,K]}{Z([K,K])}\in \mathfrak{E}_{<\infty}.
\end{equation}

\medskip

\textbf{Claim 1}.  If $S$ is a compact connected semisimple group with $S/Z(S)\in \mathfrak{E}_{<\infty}$, then $S\in \mathfrak{E}_{<\infty}$.

\medskip

In fact, by Lemma \ref{lem:lie-components}, $S/Z(S)$ is, up to topological isomorphism, the cartesian product $\prod_{j \in J}L_j^{\kappa_j}$ of the so-called Lie components $L_j$ each appearing with a certain multiplicity $\kappa_j$. If $S/Z(S)\in\mathfrak{E}_{<\infty}$ then every such $L_j$ has finite multiplicity: otherwise $S/Z(S)$ would have a direct factor isomorphic to an infinite power $L^{\mathbb{N}}$ of an infinite compact group $L$, on which the Bernoulli shift $\varphi_L$ has $\mathsf{h}_{\mathsf{top}}(\varphi_L)= +\infty$ (see Example \ref{ex:Bernoulli}). This goes in contradiction with $S/Z(S)\in\mathfrak{E}_{<\infty}$, because extending $\varphi_L$ by the identity on the complementary direct factor yields a continuous endomorphism of $S/Z(S)$ with infinite topological entropy. Moreover, $Z(S)$ is zero-dimensional, because it is totally disconnected by \cite[Theorem 9.19 (ii)]{hofmor}. By \cite[Theorem 6.9]{DikranSanchis}, every quotient subgroup of $S$ belongs to $\mathfrak{E}_{<\infty}$. Claim 1 follows.

\medskip

Consider $S=[K,K]$, applying  Claim 1, which shows that $[K,K]=\overline{[G,G]}\in\mathfrak{E}_{<\infty}$. The first part of the result is proved.

\medskip

\textbf{Claim 2.} There is a centerfree profinite group $H$ with $H/Z(H) \in \mathfrak{E}_{<\infty}$ and $\overline{[H,H]}\not\in \mathfrak{E}_{<\infty}$.

\medskip

Assume $G_0$ is properly contained in $G$. We construct a counterexample with a centerfree profinite group $H$ (where $H_0=1$). Let $p_1,p_2,\ldots$ be a sequence of pairwise distinct odd primes and, for every $n \in \mathbb{N}$, let $d_n$ be the minimal positive integer such that $p_n \mid 2^{d_n}-1$, hence $\gcd(d_n, p_n)=1$. Consider the finite field $\mathbb{F}_{2^{d_n}}$ of order $2^{d_n}$; then, its multiplicative group $\mathbb{F}_{2^{d_n}}^\times$ (cyclic of order $2^{d_n}-1$) contains an element $\lambda_n$ of order $p_n$. Set the finite discrete abelian groups
\begin{equation}
    A_n\coloneqq (\mathbb{F}_{2^{d_n}},+)\simeq\mathbb{Z}(2)^{d_n} = \underbrace{\mathbb{Z}(2) \times \ldots \times \mathbb{Z}(2)}_{d_n\mbox{-times}}  \quad \mbox{and} \quad B_n \coloneqq \langle t_n \rangle \simeq \mathbb{Z}(p_n),
\end{equation}
and define the semidirect product group 
\begin{equation}
D_n\coloneqq A_n \rtimes B_n,
\end{equation}
by the action of $B_n$ on $A_n$ given by
\begin{equation}
t_n a_nt_n^{-1}\coloneqq \lambda_n a_n,
\end{equation}for all $a_n\in A_n$. We have $D_n/A_n\simeq B_n$, and $[D_n,D_n]=[A_n,B_n]$ where $[a_n,t_n]=(\lambda_n^{-1}-1)a_n$ for all $a_n\in A_n$, and since $\lambda_n^{-1}-1\neq0$ we get 
\begin{equation}
[D_n, D_n]=A_n.    
\end{equation}
It can be checked that $D_n$ is centerfree:
\begin{equation}
Z(D_n)=1,
\end{equation}
in fact, if $z_n=a_nt_n^k\in Z(D_n)$, centrality of $z_n$ with respect to $A_n$ implies 
$k\equiv0\bmod p_n$, and centrality of $z_n=a_n$ with respect to $t_n$ yields $a_n=0$.
We have that 
\begin{equation}
    H\coloneqq\prod_{n \in \mathbb{N}} D_n
\end{equation}
is a centerfree  profinite  group, in particular $H/Z(H)=H$ is compact. Moreover, $H \in \mathfrak{E}_{< \infty}$, and actually $H \in \mathfrak{E}_0$. 

To check this, we first observe that for any pair of positive integers $m \neq n$ there are just trivial homomorphisms from $D_m$  to $ D_n$:
\begin{equation}\label{eq:nocrosshomos}
\mathrm{Hom}(D_m,D_n)=1\qquad\textup{whenever }m\neq n.
\end{equation}
Indeed, if $f\colon D_m\to D_n$ is a homomorphism, composing $f$ with the canonical projection $D_n\to D_n/A_n\simeq B_n$ gives a homomorphism from $D_m$ to the abelian group $B_n$, and hence it factors through the abelianization $D_m/[D_m,D_m]\simeq B_m$. Since $B_m$ and $B_n$ have distinct prime orders, this homomorphism is trivial. Therefore $f(D_m)\subseteq A_n$.  As $A_n$ is an abelian $2$-group, repeating a similar argument on $f$, we get that $f$ is equal to the trivial homomorphism.

Let now $\varphi\in\mathrm{End}(H)$. Denote by $\pi_n\colon H\to D_n$ the projection onto the $n$-th coordinate. The homomorphism $\pi_n\circ\varphi\colon H\longrightarrow D_n$ is continuous into the finite discrete group $D_n$, hence it depends on only finitely many coordinates. By  \eqref{eq:nocrosshomos}, for every $n\in\mathbb{N}$, we have $\pi_n\circ\varphi=\varphi_n\circ \pi_n$ for some $\varphi_n\in\mathrm{End}(D_n)$. It follows that every continuous endomorphism of $H$ is coordinatewise 
\begin{equation}
\varphi=\prod_{n\in\mathbb{N}}\varphi_n.
\end{equation}
For every finite subset $F\subset\mathbb{N}$, put
\begin{equation}
V_F \coloneqq \prod_{n\in F}\{1_n\} \times \prod_{n\notin F}D_n.
\end{equation}
By varying $F$, the groups $V_F$ form a local basis of neighborhoods of the identity of $H$ consisting of compact open subgroups. Moreover, every $V_F$ is fully invariant, since every continuous endomorphism of $H$ acts coordinatewise. For all $\varphi\in\mathrm{End}(H)$ and for all $F$, we have $\varphi(V_F)\subseteq V_F$, thus $C_k(\varphi,V_F)=V_F$ for all $k\geq1$, and $\mathsf{H}_{\mathsf{top}}(\varphi,V_F)=0$. By van Dantzig's Theorem (see Remark \ref{rem:vanDantzig}), we get $\mathsf{h}_{\mathsf{top}}(\varphi)=0$ for all $\varphi\in\mathrm{End}(H)$, namely,
\begin{equation}
H\in\mathfrak{E}_0\subseteq \mathfrak{E}_{<\infty}.
\end{equation}
It remains to examine the closed derived subgroup $[H,H]$. Since $[D_n,D_n]=A_n$ then $[H,H]\subseteq \prod_{n\in\mathbb{N}}A_n$; the converse inclusion holds true because every element of $A_n$ is a commutator $[a_n,t_n]$ for suitable $a_n\in A_n$, hence every $x=(x_n)_{n\in \mathbb{N}}\in \prod_{n\in\mathbb{N}}A_n$ is of the form $x=[(a_n)_{n\in \mathbb{N}},(t_n)_{n\in \mathbb{N}}]$. Therefore \begin{equation} \overline{[H,H]}=[H,H]=\prod_{n \in \mathbb{N}} A_n\simeq \prod_{n \in \mathbb{N}} \mathbb{Z}(2)^{d_n} \simeq \left(\mathbb{Z}(2)^{\mathbb{N}}\right)^{\mathbb{N}}.
\end{equation}
This allows us to define the Bernoulli shift on $\overline{[H,H]}$ that has infinite topological entropy (see Example \ref{ex:Bernoulli}), hence $\overline{[H,H]} \not\in \mathfrak{E}_{<\infty}$.
\end{proof}

\begin{remark}
In the claim of the proof of Theorem \ref{1stmain} we can see that actually: if $S$ is a compact connected semisimple group with $S/Z(S)\in \mathfrak{E}_{<\infty}$, then $S\in \mathfrak{E}_0$. In fact this shows the argument and analogies with \cite[Proof of Theorem 6.9]{DikranSanchis}, a modification of which shows that for a compact connected group $S$ we have that every quotient subgroup of $S$ belongs to $\mathfrak{E}_0$ if and only if $S$ is semisimple and every Lie component has finite multiplicity. 
In particular, this is to say that: if $G$ is a locally compact connected group with $G/Z(G)$ compact in $\mathfrak{E}_{<\infty}$, then $\overline{[G,G]}$ is compact and  belongs to $\mathfrak{E}_0$. 
\end{remark}

We briefly recall a construction due to van Est \cite{EstAlg} and recently studied also by Peng and others \cite{Peng}. Let $G$ be a locally compact group and let $\mathrm{Aut}(G)$ denote the group of topological automorphisms of $G$. If $H$ is a subgroup of $\mathrm{Aut}(G)$, one may endow $H$ with the \textit{Birkhoff topology}, which is the group topology induced by the local basis at the identity  \begin{equation}\label{bt}\mathcal S(C,U_0)\coloneqq\{\varphi\in H \mid \textup{ for all } g\in C,\  \varphi(g)\in gU_0,  \ \varphi^{-1}(g)\in gU_0\},\end{equation} where $C$ is a compact subset and $U_0$ a neighborhood of the identity of $G$. Following van Est \cite{EstAlg}, a locally compact group $G$ is called $\mathsf{CA}$-\textit{group} if the subgroup of inner automorphisms
$\mathrm{Inn}(G)$ is closed in $\mathrm{Aut}(G)$, where both groups are endowed with the Birkhoff topology. Typical examples of $\mathsf{CA}$-groups include compact groups, connected nilpotent Lie groups and connected semisimple Lie groups; see \cite{EstAlg,Peng} for details. The relation between  $G/Z(G)$ and $\mathrm{Inn}(G)$ in a locally compact group $G$ is given by the natural continuous bijective homomorphism
\begin{equation}\label{eq:isocontiG/ZInn}
xZ(G) \in G/Z(G)\mapsto \psi_x \in \mathrm{Inn}(G),\quad \textup{where}\quad \psi_x(g)=xgx^{-1}.
\end{equation}
A connected Lie group is a $\mathsf{CA}$-group if and only if \eqref{eq:isocontiG/ZInn} is a topological isomorphism, see \cite{Peng}. If $G$ is a $\mathsf{Z}$-group, then 
the continuous isomorphism \eqref{eq:isocontiG/ZInn} is automatically a homeomorphism, thus $G$ is a $\mathsf{CA}$-group. 
The connection \eqref{eq:isocontiG/ZInn} provides a natural framework for interpreting Theorem \ref{1stmain} in dynamical terms: $G/Z(G)\in\mathfrak{E}_{<\infty}$ may equivalently be viewed as $\mathrm{Inn}(G)\in\mathfrak{E}_{<\infty}$, and
conditions of finiteness on the topological entropy of the central quotient may be viewed as restrictions on the dynamical behavior of the inner automorphisms of a locally compact group. The class of $\mathsf{CA}$-groups therefore provides a broader setting in which to investigate entropy versions of Schur-type phenomena. Nevertheless, the $\mathsf{CA}$-property alone is not sufficient: every compact group is a $\mathsf{CA}$-group, whereas the counterexample constructed in the proof of Theorem \ref{1stmain} is profinite hence compact. Motivated by these observations and by the results of \cite{RussoWaka, RussoWakabis, olwethuadditional}, it is natural to ask whether Theorem \ref{1stmain} extends beyond the class of connected $\mathsf{Z}$-groups.

\begin{problem}\label{problem} Investigate conditions on a $\mathsf{CA}$-group $G$ such that \[G/Z(G)\in \mathfrak{E}_{<\infty} \ \ \ \mbox{implies} \ \ \  \overline{[G,G]} \in \mathfrak{E}_{<\infty}.\] 
\end{problem}

The preceding discussions suggest further generalizations. For instance, if $G$ is a locally compact group possessing a nontrivial fully invariant closed subgroup $N$, we can define the restriction map 
\begin{equation}\label{res}
    \mathrm{res}_N \colon \varphi \in  \mathrm{End}(G) \mapsto \mathrm{res}_N(\varphi)\coloneqq\left.\varphi\right\rvert_N \in \mathrm{End}(N).
\end{equation}
If $\mathrm{res}_N$ is surjective, any continuous endomorphism $\psi$ of $N$ arises from the restriction of a continuous endomorphism of $G$ to $N$ (i.e. $\psi$ extends to a continuous endomorphism of $G$). Under these lucky hypotheses, $G\in\mathfrak{E}_{<\infty}$ implies $N\in\mathfrak{E}_{<\infty}$ by Lemma \ref{teorspezzamento}(i); the same holds with $\mathfrak{E}_0$ in place of $\mathfrak{E}_{<\infty}$. Conversely, to infer $G\in\mathfrak{E}_{<\infty}$ from $N\in\mathfrak{E}_{<\infty}$ one also needs control of $G/N$ and an applicable Addition Theorem. Although the implication from $G\in\mathfrak{E}_{<\infty}$ to $ N\in\mathfrak{E}_{<\infty}$, guaranteed by the surjectivity of \eqref{res}, is not itself a dynamical Schur's Theorem (with $N=\overline{[G,G]}$), it provides essential ingredients: one would need suitable conditions for the finiteness of the topological entropy to be lifted from $G/Z(G)$ to $G$, and so Problem \ref{problem} would follow. The extension problem is particularly natural for classes of locally compact groups whose endomorphism monoids admit an explicit structural description. Certain Heisenberg groups may provide useful test cases in this direction.
The counterexample $H$ in Theorem \ref{1stmain}, however, shows that full invariance alone is not sufficient: $H\in\mathfrak{E}_0$ does not imply $\overline{[H,H]}\in \mathfrak{E}_{<\infty}$, and although $\overline{[H,H]}$ is fully invariant, not all of its continuous endomorphisms extend to $H$.

\begin{problem}\label{problextres} \

(i). If $N$ is a closed fully invariant subgroup of a locally compact group $G$, characterize the pairs $(G,N)$ for which \eqref{res} is surjective, in particular for $N=\overline{[G,G]}$.

(ii). Investigate  Schur's Theorem for locally compact groups possessing large fully invariant closed subgroups  in $\mathfrak{E}_{< \infty}$ such that \eqref{res} is surjective.
\end{problem}

%%%%%%%%%%%%%%%%%%%%%%%%%%%%%%%%%%%%%%%%%%%%%%%%%%%%%%%%%%%%%%%
\section{The second main theorem and the interesting behavior of Heisenberg groups}\label{sez:ourTHEOR1.2}

We begin by recalling the definition and some characterizations of Heisenberg groups; these are suitable families of upper triangular matrices, see \cite{hofmor, rob, RussoWaka}.

\begin{definition}[See \cite{BonDikHeisenberg,RussoWaka}, Large Heisenberg Groups]\label{classicalconstruction}
Let $R$ be a nontrivial topological commutative ring with identity, and let $n\in\mathbb{N}$. The \emph{$n$-th Heisenberg group over $R$} is denoted by $\mathbb{H}_n(R)$ and is the topological group of the $(n+2)\times(n+2)$ matrices
\begin{equation}\label{largeheisenberg1} 
M(\mathbf{a},\mathbf{b},c) \coloneqq \left(\begin{array}{c|cccc|c} 1 & a_1 & a_2 & ... & a_n & c\\
\hline 0 & 1 & 0 & ... & 0 & b_1  \\ 
0 & 0 & 1 & ... & 0 & b_2  \\ 
... & ... & ... & ...  & ... &...  \\
0 & 0 & 0 & ... & 1 & b_n  \\ 
\hline 0 & 0 & 0 & 0 & 0 & 1\\ \end{array}\right)= 
\begin{pmatrix}
    1 & \mathbf{a} & c \\
    \mathbf{0}^\top & I_n & \mathbf{b}^\top \\
    0 & \mathbf{0} & 1 
\end{pmatrix},
\end{equation}
where $\mathbf{a}=(a_1,\dots,a_n),$ $\mathbf{b}=(b_1,\dots,b_n),$ $\mathbf{0}=(0,\dots,0)\in R^n$, $\mathbf{b}^\top$ is the transpose of $\mathbf{b}$ and $c\in R$. The group \begin{equation}\mathbb{H}_n(R)\coloneqq\{M(\mathbf{a},\mathbf{b},c) \mid a_i, b_j, c \in R; \ i,j \in \{1, 2, \ldots, n\} \}\end{equation} is topologized with the induced topology, which is obtained by the product topology of finitely many copies of $R$.
\end{definition}

\begin{remark}\label{wellknownfacts} 
As topological spaces, we may identify
\begin{equation}
\mathbb{H}_n(R)\simeq R^n\times R^n\times R,
\end{equation}
since the elements $M(\mathbf{a},\mathbf{b},c)$ are in bijective correspondence with $(\mathbf{a},\mathbf{b},c)$. The product of two matrices in $\mathbb{H}_n(R)$ is given by
\begin{equation}\label{eq:prodottoHeisen}
M(\mathbf{a},\mathbf{b},c) \ M(\mathbf{u},\mathbf{v},w)=M(\mathbf{a}+\mathbf{u},\,\mathbf{b}+\mathbf{v},\, c + w+ \mathbf{a} \mathbf{v}^\top),
\end{equation}
where $\mathbf{a}\mathbf{v}^\top = a_1 v_1 + \ldots +  a_n v_n$ according to the ring operations and this sum originates from a bilinear form over $R$. In particular, one can note from the $(1,n+2)$ entry  $c+w+\mathbf{a} \mathbf{v}^\top$ that the Heisenberg group is nonabelian.  Indeed, the center of the Heisenberg group is
\begin{equation}\label{eq:centreHei}
Z\left(\mathbb{H}_n(R)\right)=\{M(\mathbf{0},\mathbf{0},c)\mid c\in R\}\simeq R^+,
\end{equation}
where $R^+$ denotes the additive topological group underlying $R$.
The inverse of a matrix in $\mathbb{H}_n(R)$ is given by
\begin{equation}\label{eq:inversoHeisen}
M(\mathbf{a},\mathbf{b},c)^{-1}=M(-\mathbf{a},-\mathbf{b},-c+\mathbf{a}\mathbf{b}^\top).
\end{equation}
The commutator between two matrices in $\mathbb{H}_n(R)$ is
\begin{equation}\label{eq:commHeis}
\left[M(\mathbf{a},\mathbf{b},c),M(\mathbf{u},\mathbf{v},w)\right]=M(\mathbf{0},\mathbf{0},\mathbf{a}\mathbf{v}^\top-\mathbf{u}\mathbf{b}^\top),
\end{equation}
therefore the derived subgroup is 
\begin{equation}\label{eq:comm=zentrum}
\left[\mathbb{H}_n(R),\mathbb{H}_n(R)\right] 
= Z\left(\mathbb{H}_n(R)\right),
\end{equation}
and this turns out to be a closed normal subgroup in $\mathbb{H}_n(R)$. Thus, the abelianization of $\mathbb{H}_n(R)$ is given by \begin{equation}\label{eq:abelianisedHnR}
\mathbb{H}_n(R)/\left[\mathbb{H}_n(R),\mathbb{H}_n(R)\right]\simeq (R^+)^{2n}.
\end{equation}
The Heisenberg group $\mathbb{H}_n(R)$ is nilpotent of nilpotency class $2$, because $\mathbb{H}_n(R)$ is not abelian, $R\neq0$, and $\mathbb{H}_n(R)/Z\left(\mathbb{H}_n(R)\right)$ is abelian. For every $m\in\mathbb{N}$, the $m$-th power of a matrix in a Heisenberg group is
\begin{equation}\label{eq:powergen}
    M(\mathbf{a},\mathbf{b},c)^m=M\left(m\mathbf{a}, m\mathbf{b},mc+ \binom{m}{2} \mathbf{a}\mathbf{b}^\top\right),
\end{equation}
as can be readily checked by induction.
If $R^+$ is topologically finitely  generated, and $Y\subseteq R$ is a finite set of topological generators of $R^+$, then a finite set of topological generators of $\mathbb{H}_n(R)$ is given by
\begin{equation}\label{eq:topgenHnR}
\left\{M(y\mathbf{e}_k,\mathbf{0},0),\ M(\mathbf{0},y\mathbf{e}_k,0) \mid \, y\in Y,\ 1\leq k\leq n  \right\},
\end{equation}
where $\mathbf{e}_k=(0,\ldots,0,1,0,\ldots,0)$ is the sequence of all zeros except for a $1$ in the $k$-th position.

Now, $E\coloneqq \{M(\mathbf{a},\mathbf{0},0) \mid \mathbf{a}\in R^n\}\simeq (R^+)^n$ and $F\coloneqq\{M(\mathbf{0},\mathbf{b},0) \mid \mathbf{b}\in R^n\}\simeq (R^+)^n$ are abelian subgroups of $\mathbb{H}_{n}(R)$ with trivial intersection, and we also find that $H_1= Z\left(\mathbb{H}_{n}(R)\right)\times E\simeq (R^+)^{n+1}$ and $H_2= Z\left(\mathbb{H}_{n}(R)\right)\times F\simeq (R^+)^{n+1}$ are maximal abelian  subgroups such that $ \mathbb{H}_{n}(R)\simeq H_1 \rtimes F\simeq H_2\rtimes E$ (see \cite{BonDikHeisenberg}).
Note also from  \eqref{eq:commHeis} that
\begin{equation}\label{eq:commHeisimp}
\overline{\left[\mathbb{H}_n(R),\mathbb{H}_n(R)\right]} = \overline{\left[E,F\right]} = Z (\mathbb{H}_n(R)).
\end{equation}
For a finite family of nontrivial commutative topological rings with identity $(R_j)_{j=1}^N$, the coordinatewise identification induces a topological group isomorphism
\begin{equation}
\mathbb{H}_n\Bigg(\prod_{j=1}^N R_j\Bigg)\simeq \prod_{j=1}^N \mathbb{H}_n(R_j).
\end{equation}
\end{remark}

Lemma \ref{p-rank-finite} is a classification theorem for topological groups; it does not provide a classification of compatible topological ring structures. Motivated by the torsionfree factors $\mathbb{Q}_p$ and $\mathbb{Z}_p$ occurring in
that lemma, we single out the following class of topological rings.
\begin{corollary}\label{p-rank-finiterings}
Let $p$ be a prime  and $\varepsilon,\zeta$ nonnegative integers such that $\varepsilon+\zeta>0$. Consider $\mathbb{Z}_p$ and $\mathbb{Q}_p$ as topological rings, with $p$-adic topology. Set
\begin{equation}
R_{p,\varepsilon,\zeta} \coloneqq \mathbb{Q}_p^\varepsilon \times\mathbb{Z}_p^\zeta    
\end{equation}
with componentwise addition and multiplication and with the product topology. Then $R_{p,\varepsilon,\zeta}$ is a nontrivial locally compact commutative ring with identity, whose additive group $R_{p,\varepsilon,\zeta}^+$ is a locally compact abelian 
$p$-group of  $\mathrm{rank}_p(R_{p,\varepsilon,\zeta}^+)= \varepsilon + \zeta < \infty$.
\end{corollary}
Like every locally compact abelian $p$-group, $R_{p,\varepsilon,\zeta}^+$ is periodic \cite[p. 49]{HHR}. Theorem \ref{2ndmain} deals with  Heisenberg groups over rings as per Corollary \ref{p-rank-finiterings}. We have all that we need in order to show our second main result. For fixed $p,\varepsilon,\zeta$, we simplify the notation by writing 
\begin{equation}
R\coloneqq R_{p,\varepsilon,\zeta} \qquad\textup{and}\qquad r\coloneqq\mathrm{rank}_p(R_{p,\varepsilon,\zeta}^+),
\end{equation}
and by omitting the understood dependence on $p,\varepsilon,\zeta$ of all subsequently defined objects.

\begin{proof}[Proof of Theorem \ref{2ndmain}]
By the structure of $\mathbb{H}_n(R)$ which we have illustrated in Definition \ref{classicalconstruction} and Remark \ref{wellknownfacts}, along with Corollary \ref{p-rank-finiterings}, $\mathbb{H}_n(R)$ is a nilpotent group of nilpotency class $2$, and it is locally compact and totally disconnected. The $p^k$-th power of any element $M(\mathbf{a},\mathbf{b},\mathbf{c})\in \mathbb{H}_n(R)$ is given by \eqref{eq:powergen} as
\begin{equation}
 M(\mathbf{a},\mathbf{b},\mathbf{c})^{p^k}
 =M\left(p^k \mathbf{a}, p^k \mathbf{b},  p^kc+\binom{p^k}{2}\mathbf{a}\mathbf{b}^{\mathsf T} \right).
\end{equation}
The first three scalar-multiple terms converge to zero according to the $p$-adic topology. If $p$ is odd $\binom{p^k}{2}$ is divisible by $p^k$, while if $p=2$ it is divisible by $2^{k-1}$; in any case the final term also converges to zero. Thus every element is a $p$-element, and $\mathbb{H}_n(R)$ is a $p$-group. Every element of the $p$-group $\mathbb{H}_n(R)$ is compact (see \cite[p. 48]{HHR} adapted for nonabelian $p$-groups), hence $\mathbb{H}_n(R)$ is a periodic group. By Remark \ref{wellknownfacts} and Corollary \ref{p-rank-finiterings} to Lemma \ref{p-rank-finite}, we get that $[\mathbb{H}_n(R), \mathbb{H}_n(R)] \simeq R^+$ and $\mathbb{H}_n(R)/Z(\mathbb{H}_n(R))\simeq (R^+)^{2n}$ are locally compact abelian $p$-groups of finite $p$-rank
\begin{equation}
\mathrm{rank}_p\big([\mathbb{H}_n(R),\mathbb{H}_n(R)]\big)=r \quad \textup{and}\quad \mathrm{rank}_p\big(\mathbb{H}_n(R) / Z(\mathbb{H}_n(R))\big) = 2nr.\end{equation}
In particular, $[\mathbb{H}_n(R), \mathbb{H}_n(R)]$ and $\mathbb{H}_n(R) / Z(\mathbb{H}_n(R))$ belong to $\mathfrak{E}_{<\infty}$ by Lemma \ref{prankfin-entropy}. A large part of Theorem \ref{2ndmain} has already been proved: it remains just
\begin{equation}\label{eq:rankHnRbis}
    \mathrm{rank}_p (\mathbb{H}_n(R)) = (2n+1)r.
\end{equation}

\medskip

\textbf{Claim 1.} Let $m \in\mathbb{N}$. Every compact subgroup $C$ of $R^m$ has minimal number of topological generators
\begin{equation}\label{formercorollary}
 d(C)\leq mr.
\end{equation}
In fact, from \cite[Exercise E1.16]{hofmor} applied componentwise to $\mathbb{Q}_p^\varepsilon$, the reader can observe that $R^m = \bigcup_{s\geq0} B_s^m$ where 
\begin{equation}
B_s\coloneqq \frac{1}{p^s}\mathbb{Z}_p^{\varepsilon} \, \times \, \mathbb{Z}_p^{\zeta},  
\end{equation}
and the $B_s^m \simeq \mathbb{Z}_p^{mr}$ form an ascending sequence of compact open subgroups of $p$-rank $mr$. Since $C$ is compact, there exists $s\geq0$ such that
\begin{equation}
C\subseteq B_s^m.
\end{equation}
As a closed subgroup of $\mathbb{Z}_p^{mr}$, we have $C\simeq \mathbb{Z}_p^t$ for some $0\leq t\leq mr$. Therefore $d(C)=t$.

\medskip

\textbf{Claim 2.}  Every topologically finitely generated closed subgroup $L$ of $\mathbb{H}_n(R)$ satisfies
\begin{equation}\label{eq:dupperbound}
 d(L)\leq(2n+1)r.
\end{equation} Let $x_1,\ldots, x_t$ be a set of topological generators of $L$. Since there are finitely many of them, there exists $s\geq0$ such that all their coordinates belong to $B_s$. We have that $K_s\coloneqq\{M(\mathbf{a},\mathbf{b},c)\mid \,\mathbf{a},\mathbf{b}\in B_s^n,\,c\in B_{2s}\}$ is a subgroup of $\mathbb{H}_n(R)$, by the product and inverse rules \eqref{eq:prodottoHeisen}, \eqref{eq:inversoHeisen} with $B_s,B_sB_s\subseteq B_{2s}$. Moreover $K_s\simeq B_s^n\times B_s^n\times B_{2s}$ as topological spaces, hence $K_s$ is compact. Therefore, $L=\overline{\langle x_1,\ldots, x_t\rangle}\subseteq K_s$ is compact ($L$ is alternatively proved to be compact through \cite[Remark 1.14]{HHR} applied to the periodic solvable group $\mathbb{H}_n(R)$). Now, consider the projection $\pi\colon M(\mathbf{a},\mathbf{b},c) \in \mathbb{H}_n(R)\longmapsto \pi(M(\mathbf{a},\mathbf{b},c))\coloneqq (\mathbf{a},\mathbf{b}) \in R^n \times R^n$. Then $\pi(L)$ and $L\cap Z(\mathbb{H}_n(R))$ are compact subgroups respectively in $R^{2n}$ and $R$, and Claim 1 provides
\begin{equation}d(\pi(L)) \leq 2nr \qquad \mbox{and}  \qquad d(L\cap Z(\mathbb{H}_n(R)))\leq r.
\end{equation}
Choose topological generators of $\pi(L)$ and lift them to $L$; add generators of $L\cap Z(\mathbb{H}_n(R))$. The closed subgroup $\widetilde{L}$ generated by these elements is compact (since $L$ is compact and $\pi(\widetilde{L})$ is closed), hence $\pi(\widetilde{L})=\pi(L)$. Moreover, $\widetilde{L}$ contains $L\cap\ker(\pi)=L\cap Z(\mathbb{H}_n(R))$. Then, for every $x\in L$ there is $y\in \widetilde{L}$ such that $\pi(x)=\pi(y)$, yielding $x\in \widetilde{L}$ and $\widetilde{L}=L$. Therefore we conclude that
\begin{equation}\label{eq:daclaim2}
 d(L)\leq d(\pi(L))+d(L\cap Z(\mathbb{H}_n(R)))\leq 
 (2n+1)r.
\end{equation}

\medskip

From Claim 2, taking the supremum of both sides of \eqref{eq:dupperbound} over all topologically finitely generated closed subgroups $L$ of $\mathbb{H}_n(R)$,
 we get the following upper bound:
\begin{equation}\label{extraequation}\mathrm{rank}_p (\mathbb{H}_n(R))  \le (2n+1) r.\end{equation}

For the reverse inequality of \eqref{extraequation}, namely for the lower bound
\begin{equation}\label{extraequationbis}\mathrm{rank}_p (\mathbb{H}_n(R))  \ge (2n+1) r,\end{equation}  it is enough to exhibit a topologically finitely generated closed subgroup of $\mathbb{H}_n(R)$ achieving defect $(2n+1)r$. We set 
\begin{equation}
A\coloneqq \mathbb{Z}_p^{\varepsilon} \times \mathbb{Z}_p^{\zeta}\subseteq \mathbb{Q}^\varepsilon_p \times \mathbb{Z}^\zeta_p= R,\end{equation}
which is a compact open subring of $R$ and a free $\mathbb{Z}_p$-module whose rank is equal to $\mathrm{rank}_p(A^+)=r$ (here a finite set generates $A$ as a $\mathbb{Z}_p$-module if and only if it topologically generates its underlying additive group $A^+$, for $\mathbb{Z}$ is dense in $\mathbb{Z}_p$). We look at
\begin{equation}
U\coloneqq \{M(\mathbf{a},\mathbf{b},c) \mid \mathbf{a},\mathbf{b}\in(pA)^n,\ c\in pA\} \ \mbox{and} \ 
 U_2=
 \{M(\mathbf{a},\mathbf{b},c) \mid \mathbf{a},\mathbf{b}\in(p^2A)^n,\ c\in p^2A\}
\end{equation}
which are compact open subgroups of $\mathbb{H}_n(R)$. Our candidate to achieve the lower bound \eqref{extraequationbis} is such a $U$. The map $\rho\colon M(\mathbf{a},\mathbf{b},c)\in U\longmapsto (\mathbf{a}+(p^2A)^n, \mathbf{b}+(p^2A)^n, c+p^2A)\in (pA/p^2A)^{2n+1}$ is a homomorphism, because $(pA)(pA)\subseteq p^2A$ and the bilinear term in a product \eqref{eq:prodottoHeisen} cancels modulo $p^2A$ (the multiplication modulo $U_2$ is additive). Furthermore, $\rho$ is surjective with kernel $U_2$, hence
\begin{equation}\label{eq:U-quotient}
 U/U_2 \simeq (pA/p^2A)^{2n+1} \simeq \mathbb{Z}(p)^{(2n+1)r}.
\end{equation}
Every set of topological generators of $U$ projects onto a generating set of $U/U_2$, therefore  $d(U)\geq (2n+1)r$. Conversely, let $f_1,\ldots,f_{r}$ be the standard $\mathbb{Z}_p$-basis of $A$ and let $\mathbf{e}_1,\ldots,\mathbf{e}_n$ denote the standard coordinate vectors of $R^n$. The elements
\begin{equation}
 M(pf_j \mathbf{e}_k,0,0),
 \qquad
 M(0,pf_j \mathbf{e}_k,0),
 \qquad
 M(0,0,pf_j),
\end{equation}
for $1\leq k\leq n$ and $1\leq j\leq r$, topologically generate $U$: the first family generates $\{M(\mathbf{a},\mathbf{0},0)\mid \mathbf{a}\in(pA)^n\}$, the second family generates $\{M(\mathbf{0},\mathbf{b},0)\mid \mathbf{b}\in(pA)^n\}$, the third one $\{M(\mathbf{0},\mathbf{0},c)\mid c\in pA\}$, and $M(\mathbf{a},\mathbf{b},c)=M(\mathbf{0},\mathbf{b},0)M(\mathbf{a},\mathbf{0},0)M(\mathbf{0},\mathbf{0},c)$. We found the desired property for $U$,
\begin{equation}
d(U)=nr+nr+r = (2n+1)r.
\end{equation}
From the definition of $p$-rank and the fact that $U$ is a topologically finitely generated closed subgroup of $\mathbb{H}_n(R)$, we get $\mathrm{rank}_p(\mathbb{H}_n(R))\geq d(U)$, and also the lower bound \eqref{extraequationbis} is proved. Combining \eqref{extraequation} and \eqref{extraequationbis}, we obtain \eqref{eq:rankHnRbis} which completes the proof.
\end{proof}

The construction of $U$ of the proof of Theorem \ref{2ndmain} applies specifically to \cite[Theorem 1.5]{RussoWaka}, where the conclusions on periodicity, finite topological entropy and finite $p$-rank of $\mathbb{H}_n(\mathbb{Q}_p)$ are correct, however the precise value of its $p$-rank (coming from \cite[Lemma 4.4]{RussoWaka}) should be corrected to $\mathrm{rank}_p(\mathbb{H}_n(\mathbb{Q}_p)) = 2n+1$, falling into the case of \eqref{eq:rankHnRbis} and Theorem \ref{2ndmain} with $\zeta=0$ and $\varepsilon=1$. We now explain how this reflects at the level of Frattini subgroup used in \cite{RussoWaka}.
\begin{definition}[Frattini Subgroup, see \cite{rz}]
For a profinite group $G$, the \emph{Frattini subgroup} $\mathrm{Frat}(G)$ of $G$ is defined as the intersection of all the maximal open subgroups of $G$.
\end{definition}
Moreover, for a profinite group $G$, $\mathrm{Frat}(G)$ coincides with the set of all nongenerators of $G$; see \cite[Lemma 2.8.1]{rz}. If $G$ is a pro-$p$ group (i.e., a profinite $p$-group), then
\begin{equation}\label{eq:frattiniComm}\mathrm{Frat}(G)=\overline{G^p \ [G,G]}\end{equation}
by \cite[Lemma~2.8.7]{rz}, where $G^{p}\coloneqq\{g^p\mid g\in G\}$. If $G$ is a topologically finitely generated profinite group, then
\begin{equation}\label{eq:frattinid}
d(G) = d\bigl(G/\mathrm{Frat}(G)\bigr)
\end{equation}
by \cite[Lemma~2.8.6]{rz} ($\mathrm{Frat}(G)$ is a closed normal subgroup of $G$); if here in addition $G$ is a pro-$p$ group, then $G/\mathrm{Frat}(G)$ can be regarded as a vector space over the field with $p$ elements and $d\bigl(G/\mathrm{Frat}(G)\bigr)$ coincides with the dimension of $G/\mathrm{Frat}(G)$ in this sense.

For a prime $p$ and $n,\zeta\in\mathbb{N}$, the Heisenberg group $\mathbb{H}_n(\mathbb{Z}_p^\zeta)$ is a topologically finitely generated pro-$p$ group, for which the standard theory in \cite{rz} and formulas \eqref{eq:frattiniComm}, \eqref{eq:frattinid} hold true. In particular, 
\begin{equation}\label{eq:minimgensHnZpzeta}
d\big(\mathbb{H}_n(\mathbb{Z}_p^\zeta)\big) =d\big(\mathbb{H}_n(\mathbb{Z}_p^\zeta)/\mathrm{Frat}\left(\mathbb{H}_n(\mathbb{Z}_p^\zeta)\right)\big) =2n\zeta.
\end{equation}
However \eqref{eq:frattinid} does not yield an analogous identity for the $p$-rank in place of the defect, because the $p$-rank tests all closed topologically finitely generated subgroups (see this in \cite[Remark 4.2 and Eq. (4.11)]{RussoWaka} for $n=\zeta=1$).

Extending the notion of Frattini subgroup to arbitrary noncompact totally disconnected locally compact groups is more delicate: maximal proper open subgroups need not exist (as illustrated by the
additive group $\mathbb{Q}_p$) and their intersection cannot in general be identified automatically with the set of topological nongenerators. Broader approaches to the Frattini subgroup for noncompact topological groups have been developed in different settings; see, for instance, \cite{frattgen2024,frattgen1980}. At the same time, the classical profinite theory remains an effective local tool, to be applied to suitable compact pro-$p$ or pro-nilpotent subgroups in the study of noncompact totally disconnected locally compact groups \cite{caprace1,caprace2}. However, the standard theory in \cite{rz} and formulas \eqref{eq:frattiniComm}, \eqref{eq:frattinid} do not hold for arbitrary noncompact totally disconnected locally compact groups.

\begin{example}
For a prime $p$ and $n,\varepsilon\in\mathbb{N}$, consider the Heisenberg group $\mathbb{H}_n(\mathbb{Q}_p^\varepsilon)$. Since $\mathbb{H}_n(\mathbb{Q}_p^\varepsilon)$ is not a profinite group, the notion of its Frattini subgroup is subtle. In particular, $\mathbb{H}_n(\mathbb{Q}_p^\varepsilon)$ is not a pro-$p$ group, so \eqref{eq:frattiniComm} cannot be applied here. The $p$-th power map on $\mathbb{H}_n(\mathbb{Q}_p^\varepsilon)$ is surjective, because for every $M(\mathbf{a},\mathbf{b},c)\in \mathbb{H}_n(\mathbb{Q}_p^\varepsilon)$ we have $M\left(\frac{\mathbf{a}}{p}, \frac{\mathbf{b}}{p}, \frac{1}{p}\left(c-\binom{p}{2}\frac{\mathbf{a}}{p}\frac{\mathbf{b}^\top}{p}\right)\right)^p=M(\mathbf{a},\mathbf{b},c)$, hence $\mathbb{H}_n(\mathbb{Q}_p^\varepsilon)^p= \mathbb{H}_n(\mathbb{Q}_p^\varepsilon)$ and
\begin{equation}
\overline{\mathbb{H}_n(\mathbb{Q}_p^\varepsilon)^p \, \bigl[  \mathbb{H}_n(\mathbb{Q}_p^\varepsilon),\mathbb{H}_n(\mathbb{Q}_p^\varepsilon)\bigr]}=\mathbb{H}_n(\mathbb{Q}_p^\varepsilon).
\end{equation}
Thus a formal use of \eqref{eq:frattiniComm} would yield a trivial Frattini quotient and cannot be used to compute either the minimum number of topological generators or the $p$-rank of $\mathbb{H}_n(\mathbb{Q}_p^\varepsilon)$. Also \eqref{eq:frattinid} does not apply to $\mathbb{H}_n(\mathbb{Q}_p^\varepsilon)$ since this group is neither compact nor topologically finitely generated (the latter by Claim 2 in the proof of Theorem \ref{2ndmain}). This discussion leads us to review the expression of the $p$-rank in \cite[Lemma 4.4 and  Theorem 1.5]{RussoWaka} for $\varepsilon=1$. In fact, the correct estimate is given by \eqref{eq:daclaim2}, and the subgroup $U$ constructed in the proof of Theorem \ref{2ndmain} shows that this contribution is genuine, yielding $d(U)=(2n+1)\varepsilon$ and 
\begin{equation}
\mathrm{rank}_p\big(\mathbb{H}_n(\mathbb{Q}_p^\varepsilon)\big) = (2n+1)\varepsilon.
\end{equation}
\end{example}

%%%%%%%%%%%%%%%%%%%%%%%%%%%%%%%%%%%%%%%%%%%%%%%%%%%%%%%

\end{document}